\documentclass[11pt,reqno]{amsart}
\usepackage{color}

\usepackage{amsmath}
\usepackage{amsfonts,amscd}
\usepackage{amssymb}
\usepackage{url}
\usepackage{hyperref}

\usepackage[english]{babel}
\usepackage{booktabs}
\usepackage{tikz}

\theoremstyle{plain}
\newtheorem{theorem}                 {Theorem}[section]

\newtheorem{lemma}        [theorem]  {Lemma}
\newtheorem{proposition}  [theorem]  {Proposition}

\theoremstyle{definition}

\newtheorem{example}      [theorem]  {Example}

\newtheorem{definition}   [theorem]  {Definition}

\newtheorem{remark}       [theorem]  {Remark}

\numberwithin{equation}{section}

\def \rn{{\mathbb R}}
\def \cn{{\mathbb C}}

\def \B{\mathcal B}

\def \E{\mathcal E}
\def \F{\mathcal F}

\def \nab#1#2{\hbox{$\nabla$\kern -.3em\lower 1.0 ex
		\hbox{$#1$}\kern -.1 em {$#2$}}}

\def \nab#1#2{\nabla_{#1}{#2}}

\def\Re{\mathfrak R\mathfrak e}
\def\Im{\mathfrak I\mathfrak m}

\def \g{\mathfrak{g}}
\def \h{\mathfrak{h}}
\def \k{\mathfrak{k}}

\def \m{\mathfrak{m}}

\DeclareMathOperator{\End}{End}

\def \glr#1{\mathfrak{gl}_{#1}(\rn)}
\def \GLC#1{\text{\bf GL}_{#1}(\cn)}
\def \glc#1{\mathfrak{gl}_{#1}(\cn)}

\def \SLR#1{\text{\bf SL}_{#1}(\rn)}
\def \SL2{\widetilde{\text{\bf SL}}_{2}(\rn)}

\def \SLC#1{\text{\bf SL}_{#1}(\cn)}
\def \slc#1{\mathfrak{sl}_{#1}(\cn)}

\def \SO#1{\text{\bf SO}(#1)}
\def \so#1{\mathfrak{so}(#1)}
\def \SOs#1{\text{\bf SO}^*(#1)}

\def \SOC#1{\text{\bf SO}(#1,\cn)}
\def \OC#1{\text{\bf O}(#1,\cn)}

\def \soc#1{\mathfrak{so}(#1,\cn)}

\def \SUs#1{\text{\bf SU}^*(#1)}

\def \U#1{\text{\bf U}(#1)}

\def \SU#1{\text{\bf SU}(#1)}
\def \su#1{\mathfrak{su}(#1)}

\def \Sp#1{\text{\bf Sp}(#1)}
\def \sp#1{\mathfrak{sp}(#1)}

\def \SpR#1{\text{\bf Sp}(#1,\rn)}

\def \SpC#1{\text{\bf Sp}(#1,\cn)}

\DeclareMathOperator{\Div}{div}

\DeclareMathOperator{\trace}{trace}

\numberwithin{equation}{section}

\allowdisplaybreaks

\begin{document}

	

\author{Thomas Jack Munn}
\address{Mathematics, Faculty of Science\\
University of Lund\\
Box 118, Lund 221 00\\
Sweden}
\email{Thomas.Munn@math.lu.se}

\title
[Minimal Submanifolds]
{Complete minimal submanifolds of the non-compact duals to the classical compact Lie groups}

\begin{abstract}
We construct new families of complete minimal submanifolds of codimension two in the non-compact duals to the classical compact Lie groups $\SOC n/\SO n$, $\SLC n/\SU n$ and $\SpC n/\Sp n$.
\end{abstract}

\maketitle

\section{Introduction and Main Results}
\label{section-introduction}
In differential geometry, the study of minimal submanifolds of a given Riemannian ambient space $(M,g)$ plays a central role. The link between minimal submanifolds and complex analysis is classical, known since the Weierstrass-Enneper  parameterisation of minimal surfaces in $\mathbb{R}^3$.
The theory of minimal submanifolds of general Riemannian manifolds has developed into a large field. Of particular relevance here is the following result by Eells and Sampson in 1964:
\begin{proposition}[\cite{Eel-Sam}]
	Let $f$ be a Riemannian immersion, then $f$ is harmonic if and only if it is minimal.
\end{proposition}

Harmonic morphisms $\phi:(M,g)\to (N,h)$ are maps between Riemannian manifolds which (locally) pull back harmonic maps to harmonic maps. They are also useful tools for the construction of minimal submanifolds of a given ambient space $M$ as shown by Baird and Eells in 1981.

\begin{theorem}[\cite{Bai-Eel}]
	Let $\phi: M \to N^n$ be a horizontally conformal map. 
	If $n=2$ then $\phi$ is a harmonic morphism if and only if the fibres of $\phi$ are minimal at regular points. If $n > 2$ then any two of the following imply the third:
	\begin{enumerate}
		\item $\phi$ is a harmonic morphism,
		\item $\phi$ is horizontally homothetic (i.e. the gradient of the conformal factor is vertical),
		\item the fibres of $\phi$ are minimal at regular points.
	\end{enumerate}
\end{theorem}
This result is particularly useful for producing minimal submanifolds of {\it codimension two}. Complex-valued harmonic morphisms were generalised to $(\lambda,\mu)$-eigenfunctions by Gudmundsson and Sakovich in 2008 (\cite{Gud-Sak-1}).

Recently, Gudmundsson and the current author have shown that $(\lambda,\mu)$-eigenfunctions  can be used to produce minimal submanifolds of codimension two.
\begin{theorem}[\cite{Gud-Mun-1}] \label{Gud-Mun-1-Thm}
Let $\phi:(M,g)\to\cn$ be a complex-valued eigenfunction on a Riemannian manifold, such that $0\in\phi(M)$ is a regular value for $\phi$.  Then the fibre $\phi^{-1}(\{0\})$ is a minimal submanifold of $M$ of codimension two.
\end{theorem}

\smallskip 

In this paper, we apply the above result in order to construct minimal submanifolds of codimension two of
 $$\SOC n/\SO n, \ \SLC n/\SU n \text{ and } \SpC n/\Sp n.$$

\begin{theorem}\label{thm-so-minimal}
Let $n \geq 3$, $V$ be a maximal isotropic subspace of $\cn^n$, $a \in V$ and $p \in \cn^n$ be non-zero elements. 
The eigenfunction  $\psi_a : \SOC{n}/\SO n \to \cn$ defined by $$ \psi_a(z) = \trace (p a^*{z} z^{*})$$ has $0$ as a regular value whenever $p$ is not isotropic. In particular, the fibre $\psi^{-1}(\{0\})$ is a complete minimal submanifold of codimension two.
\end{theorem}

\begin{theorem}\label{thm-slc-minimal}
	Let $a$ and $p \in \cn^n$ be non-zero elements. The eigenfunction  $\psi_a  : \SLC{n}/\SU{n} \to \cn $ defined by $$ \psi_a(z)=\trace (p a^* {z} z^{*})$$ has a nowhere-vanishing gradient. The point $0$ is in its image if and only if $p$ and $a$ are linearly independent. In particular, the fibre $\psi^{-1}(\{0\})$ is a complete minimal submanifold of codimension two.
\end{theorem}

	\begin{theorem} \label{thm-Sp-minimal}
	Let $a,b,p \in \cn^n$ such that $p \neq 0$ and $(a,b) \neq (0,0)$ and consider the eigenfunction $\psi_{ab}:\SpC{n}/\Sp{n} \to \cn$ defined by $$ \ \psi_{ab} = \trace\left( \begin{pmatrix}
	a \\ b
\end{pmatrix} \begin{pmatrix}
p^* & 0
\end{pmatrix}  q q^*\right).$$
Then $\psi_{ab}$ has a nowhere-vanishing gradient and furthermore, there exists a $q$ such that
$\psi_{ab}(q)=0$
if and only if either
\begin{enumerate}
	\item the vectors $p$ and $a$ are linearly independent, or
	\item $b$ is non-zero.
\end{enumerate}
 Whenever either of these conditions are satisfied, the fibre over $0$ is a complete minimal submanifold of codimension two.
\end{theorem}

\section{Eigenfunctions and Eigenfamilies}
Let $(M,g)$ be an $m$-dimensional Riemannian manifold and $T^{\cn}M$ be the complexification of the tangent bundle $TM$ of $M$. We extend the metric $g$ to a complex bilinear form on $T^{\cn}M$.  Then the gradient $\nabla\phi$ of a complex-valued function $\phi:(M,g)\to\cn$ is a section of $T^{\cn}M$.  In this situation, the well-known complex-linear {\it Laplace-Beltrami operator} (alt. {\it tension field}) $\tau$ on $(M,g)$ acts locally on $\phi$ as follows:
$$
\tau(\phi)=\Div (\nabla \phi)=\sum_{i,j=1}^m\frac{1}{\sqrt{|g|}} \frac{\partial}{\partial x_j}
\left(g^{ij}\, \sqrt{|g|}\, \frac{\partial \phi}{\partial x_i}\right).
$$
For two complex-valued functions $\phi,\psi:(M,g)\to\cn$ we have the following well-known fundamental relation
\begin{equation*}\label{equation-basic}
\tau(\phi\, \psi)=\tau(\phi)\,\psi +2\,\kappa(\phi,\psi)+\phi\,\tau(\psi),
\end{equation*}
where the complex bilinear {\it conformality operator} $\kappa$ is given by $$\kappa(\phi,\psi)=g(\nabla \phi,\nabla \psi).$$  Locally this satisfies 
$$\kappa(\phi,\psi)=\sum_{i,j=1}^mg^{ij}\cdot\frac{\partial\phi}{\partial x_i}\frac{\partial \psi}{\partial x_j}.$$

\begin{definition}[\cite{Gud-Sak-1}]\label{definition-eigenfamily}
Let $(M,g)$ be a Riemannian manifold. Then a complex-valued function $\phi:M\to\cn$ is said to be an {\it eigenfunction} if it is eigen both with respect to the Laplace-Beltrami operator $\tau$ and the conformality operator $\kappa$; that is, there exist complex numbers $\lambda,\mu\in\cn$ such that $$\tau(\phi)=\lambda\cdot\phi\ \ \text{and}\ \ \kappa(\phi,\phi)=\mu\cdot \phi^2.$$	
A set $\E =\{\phi_i:M\to\cn\ |\ i\in I\}$ of complex-valued functions is said to be an {\it eigenfamily} on $M$ if there exist complex numbers $\lambda,\mu\in\cn$ such that for all $\phi,\psi\in\E$ we have 
$$\tau(\phi)=\lambda\cdot\phi\ \ \text{and}\ \ \kappa(\phi,\psi)=\mu\cdot \phi\,\psi.$$ 
\end{definition}

Readers interested in learning more about the theory of harmonic morphisms can consult the book \cite{Bai-Woo-book} or the online bibliography \cite{Gud-bib}.

Eigenfamilies were originally used in \cite{Gud-Sak-1} to construct harmonic morphisms. More recently, they have been studied as objects of independent interest, see, for example, \cite{Gud-Sif-Sob-2},\cite{Gud-Sob-1},\cite{Rie-1},\cite{Rie-Sif-1},\cite{Mun-Rie-1}, \cite{Gha-Gud-4}, \cite{Gha-Gud-5}.
Of particular relevance here is the recent paper \cite{Gud-Mun-1}, which introduces a method of producing minimal submanifolds of codimension two using eigenfunctions.
\begin{theorem}[\cite{Gud-Mun-1}] \label{Gud-Mun-1-Thm}
Let $\phi:(M,g)\to\cn$ be a complex-valued eigenfunction on a Riemannian manifold, such that $0\in\phi(M)$ is a regular value for $\phi$.  Then the fibre $\phi^{-1}(\{0\})$ is a minimal submanifold of $M$ of codimension two.
\end{theorem}
This theorem has been used to construct minimal submanifolds of some compact
\cite{Geg-Gud-1} and non-compact \cite{Gud-Lar-1} Riemannian symmetric spaces.

\begin{remark}
	Whenever $(M,g)$ is complete, the minimal submanifolds produced in this manner will be complete, since $\phi^{-1}(\{0\})$ will be a closed embedded submanifold. This is the case in this paper since every Riemannian symmetric space is complete.
\end{remark}

\section{Non-Compact Duals to the Classical Compact Lie Groups}
\label{section-symmetric-spaces}

The classical compact simple Lie groups $\SO n$, $\SU n$, and $\Sp n$ are all Riemannian symmetric spaces when equipped with their standard bi-invariant metrics induced by their negative Killing forms. In this section we recall this construction and their non-compact dual symmetric spaces.

 Let $H$ be one of the compact groups $\SO n$, $\SU n$, or $\Sp n$ with Lie algebra $\h$.
 
 Then the group $U=H\times H$ acts transitively by isometries on $H$ by
 $$(h_1,h_2) \cdot x = h_1 x h_2^{-1}.$$

The isotropy subgroup $K =\Delta$ is the diagonal 
$$\Delta =\{(h,h)\in H\times H\, |\, h\in H\}.$$
Thus, we identify $$ U/K = (H \times H)/\Delta \simeq H.$$

The involution is given by
$$ \sigma: (p,q) \mapsto (q,p)$$
with differential
$$ d\sigma (X,Y) \mapsto (Y,X).$$
If $B$ is an orthonormal basis for $\h $ then there are bases for the $\pm 1$-eigenspaces of $d\sigma$ given by
$$B_\k = \{ (X,X) \mid X \in B\}, \ \ \ B_{\m} = \{ (X,-X) \mid X \in B\},\\$$ 
so we have 
$$ \mathfrak u = \k \oplus \m.$$
Then we define
$$ \g = \k \oplus i \m \subset \mathfrak u^\cn $$
which is the Lie algebra of the non-compact dual group. The group $G$ is the connected subgroup of the complexification 
$$ G^\cn = U^\cn = (H \times H)^\cn$$ and $G/K$ is the non-compact dual of $U/K$. Explicitly, there is the embedding
$$ \iota : H^\cn \to G^\cn, \ \ z \mapsto (z,(z^*)^{-1})$$
whose image is $G$. Its restriction to $H$ is 
$$ h \mapsto (h,h)$$
which is identified with $\Delta$ and so the non-compact symmetric space
$$ G/K \simeq H^\cn/H.$$

The non-compact dual Riemannian symmetric spaces of $\SO n$, $\SU n$, and $\Sp n$ are $\SOC n/\SO n$, $\SLC n/\SU n$ and $\SpC n/\Sp n$ respectively.

\section{Dual Eigenfunctions}
\label{section-dual-eigenfunction}
The duality between a compact symmetric space $U/K$ and its non-compact dual $G/K$ induces a duality between $(\lambda,\mu)$-eigenfamilies on the two spaces. This was first considered in the case of harmonic morphisms, which are $(0,0)$-eigenfunctions, in \cite{Gud-Sve-1} and \cite{Gud-Sve-2} by extending the (real-analytic) $(\lambda,\mu)$-eigenfunction on $G/K$ to a holomorphic $K$-invariant function on the complexification $G^{\cn}$ via analytic continuation. This holomorphic function restricts to a real-analytic $(-\lambda,-\mu)$-eigenfunction on $U/K$.

In this work we explicitly construct the dual eigenfunctions using the Peter-Weyl theorem and Weyl's unitary trick.

\begin{proposition}[Weyl's unitary trick \cite{Knapp-Book}]
	Let $G$ be an analytic subgroup of complex matrices whose linear Lie algebra $\g$ is semisimple and stable under the map $\theta$ given by negative conjugate transpose. Let $\g = \k \oplus i\m$ be the Cartan decomposition of $\g$ defined by $\theta$ and suppose that $\k \cap \m = 0$. Let $U$ and $G^\cn$ be the analytic subgroups of matrices with respective Lie algebras $\mathfrak{u} = \k \oplus \m$ and $\g^\cn = (\k \oplus i \m)^\cn$. The group $U$ is compact. Suppose that $U$ is simply connected. If $V$ is any finite-dimensional complex vector space, then a representation of any of the following kinds on $V$ leads, via the formula
	$$ \g^\cn = \g \oplus i \g = \mathfrak{u} \oplus i \mathfrak{u},$$
	to a representation of each of the other kinds. Under this correspondence invariant subspaces and equivalences are preserved:
	\begin{itemize}
		\item a representation of $G$ on $V$,
		\item a representation of $U$ on $V$,
		\item a holomorphic representation of $G^\cn$ on V,
		\item a representation of $\g$ on $V$,
		\item a representation of $\mathfrak{u}$ on $V$,
		\item a complex-linear representation of $\g^\cn$ on $V$.
	\end{itemize}
\end{proposition}

\begin{theorem}[Peter-Weyl Theorem \cite{Knapp-Book}]
	If $H$ is a compact group, then the linear span of all matrix coefficients for all finite-dimensional irreducible unitary representations is dense in $L^2(H)$.
\end{theorem}

We now give an overview of this method applied to an arbitrary compact symmetric space $U/K$ before employing it in the cases of $\SO n$, $\SU n$, and $\Sp n$.

Let $\phi$ be a $(\lambda,\mu)$-eigenfunction on $U/K$. 
Since $U/K$ is compact, each Laplace eigenspace is finite-dimensional and so there exists a finite-dimensional unitary representation $\rho:U \to U(V)$ with orthonormal basis $\{e_i\}$ of $V$ such that we may write $\phi$ as some $K$-invariant linear combination of matrix coefficient functions with respect to this basis:
$$ \phi(u) =\sum_{i,j} c_{ij} \langle \rho_U(u) e_i, e_j \rangle .$$ 
Then we can define the function $\psi: G \to \cn$ by
$$ \psi(g) =\sum_{i,j} c_{ij} \langle \rho_G(g) e_i, e_j \rangle .$$
Then $\psi$ is clearly $K$-invariant by construction and a computation (see
\cite{Gud-Sve-1}) shows that it is a $(-\lambda,-\mu)$-eigenfunction.

\section{Dual Eigenfunctions on the Classical Non-Compact Duals}
In this section we explicitly describe the duality principle for eigenfunctions on the classical compact Lie groups. Once this is established we prove some preliminary results on the zeros and critical points of the eigenfunctions on the non-compact duals.

Let $\phi: H \simeq U/K \to \cn$ be a $(\lambda,\mu)$-eigenfunction given by matrix coefficients of some finite-dimensional unitary representation $\rho: H \to U(V)$,
$$ \phi(x) = \trace_V A\rho(x)$$
for some $A \in \End (V)$. 
 We can lift $\phi$ to a $K$-invariant function $\hat \phi$ from $U =H \times H$ by $$\hat{\phi} (x_1 ,x_2) = \phi(x_1 \cdot x_2^{-1}) = \trace_V A \rho(x_1 x_2^{-1}),$$ which then extends holomorphically to $G^\cn$ by
 $$ \hat \phi^\cn = \trace_V(A \rho^\cn(z_1 z_2^{-1})).$$
 
The map 
$$ \iota : H^\cn \to G^\cn, \ \ z \mapsto (z,(z^*)^{-1})$$
 embeds $H^\cn$ into $(H \times H)^\cn$, so we define the eigenfunction $\psi: H^\cn  \to \cn$ by
\begin{eqnarray*}
\psi(z) &=& \hat \phi^\cn(z,(z^*)^{-1}) \\&=& \trace_V A  \rho^\cn(z) \rho^\cn(z^*).
\end{eqnarray*}

Since $(\rho,V)$ is a unitary representation of $H$ we have that $\rho^\cn(h)^* =\rho^\cn(h)^{-1} $ and so
$$\psi(zh) = \trace_V A  \rho^\cn(z) \rho^\cn(h) \rho^\cn(h^*) \rho^\cn(z^*) = \trace_V A  \rho^\cn(z) \rho^\cn(z^*) = \psi(z),$$
i.e., $\psi$ defines a function on $G/K \simeq H^\cn/H$.

{\it From now on we only consider one representation at a time, so, for simplicity of notation we will no longer explicitly write $\rho^\cn(z)$ for the representation of a group element and instead simply write $z$.}

\subsection{Zeros of induced functions} \label{subsection-zeros}
Let
$$ \psi(z) = \trace A z z^* $$
be an eigenfunction such that $A = u v^*$ is a rank-one matrix. Then

\begin{eqnarray*}
\trace(Azz^*) &=& \trace (u v^* z z^*) \\
&=& \trace ( v^* z z^* u) \\
&=& v^* z z^* u.
\end{eqnarray*}
Notice that $z z^*$ is a positive definite Hermitian matrix, so it defines an inner product. In particular, if $v$ and $u$ are linearly dependent then $\psi(z)$ never vanishes. Conversely, if they are linearly independent, and $n \geq 3$, then $\psi$ must have a zero in $\SOC{n}$.

\begin{lemma}\label{lemma-zeros-sonC}
	Let $u,v \in \mathbb{C}^n$ be linearly independent vectors for $n \geq 3$. Then there exists $g \in \SOC{n}$ such that $gu$ and $gv$ are orthogonal with respect to the standard Hermitian inner product.\end{lemma} 
\begin{proof}
	First recall that $\OC{n}$ is the isometry group of the complex bilinear inner product on $\cn^n$, so by Witt's theorem it is sufficient to define an isometry between $\text{span}\{u,v\}$ and its image $\text{span}\{x=gu,y=gv\}$ with the property that $x$ and $y$ are Hermitian-orthogonal.	
	
	 For notational convenience we only display the first coordinates of the vectors in $\cn^n$, as the other entries are all zero.
	 
	Now we show that such $x,y$ always exist. Let 
	\begin{eqnarray*}
		&a = u^t u, \ b = v^t v \text{ and } c = u^t v.&
	\end{eqnarray*}
	
	We will construct $x,y$ satisfying
	\begin{eqnarray}\label{eqn-xy}
		&a = x^t x, \ b = y^t y, \ c = x^t y& \text{and } x^*y=0.
	\end{eqnarray}
	We proceed by cases. \\

	First suppose that $a = b =c = 0$. Note that this case can only occur when $n \geq 4$, since the maximal isotropic subspace in $\cn^3$ is one-dimensional. Let

$$ x = \begin{pmatrix}
		1 \\ i \\ 0 \\ 0
	\end{pmatrix} \text{ and }  \ y = \begin{pmatrix}
		0 \\ 0 \\ 1 \\ i
	\end{pmatrix}$$
	which clearly satisfy equations (\ref{eqn-xy}).
	
	Now consider the case when $a = 0$ but $(b,c) \neq (0,0)$. Then choose $s \in \cn$ with $s^2=b$ and define
	$$ x = \begin{pmatrix}
		1 \\ i \\ 0
	\end{pmatrix} \text{ and }  \ y = \begin{pmatrix}
		\tfrac{c}{2} \\ -\tfrac{ic}{2}\\ s
	\end{pmatrix}.$$
A simple calculation shows $x$ and $y$ satisfy equations (\ref{eqn-xy}) and since $s \neq 0$ or $c \neq 0$, $x$ and $y$ are clearly linearly independent. By symmetry, an identical argument works when $b=0$ and $(a,c) \neq (0,0)$.

	Finally we need to consider the case when neither $a$ nor $b$ are zero. Choose $r \in \cn$ with $r^2 = a$ and $t \in \rn \setminus \{0\}$ then let 
	
		$$ x = r \cdot \begin{pmatrix}
		\cosh t \\ i \sinh t \\ 0
	\end{pmatrix} \text{ and }  \ \tilde y = \begin{pmatrix}
		i\sinh t\\ \cosh t\\ 0
	\end{pmatrix}.$$ 
	A computation shows
 \begin{eqnarray*}
 	& x^t x = r^2 = a, \  \tilde y^t \tilde y =1, \  x^t \tilde y= 2i r \cosh t \sinh t = i r \sinh 2t&, \ x^* \tilde y=0.
 \end{eqnarray*}
 Since neither $r$ nor $t$ are zero we have that $x^t \tilde y \neq 0$ so we can pick some $s \in \cn$ such that $s^2 = b-\frac{c^2}{(x^t \tilde y)^2}$ and finally define
 $$ y = \tfrac{c}{x^t \tilde y} \cdot \tilde y + s \cdot \begin{pmatrix}
 	0 \\ 0 \\ 1
 \end{pmatrix}.$$
 It is then clear that $x$ and $y$ satisfy equations (\ref{eqn-xy}).
 
 So, by Witt's theorem there always exists a $G \in \OC{n}$ with the required properties. If $\det G =1$ let $g = G$. Otherwise, if $\det G = -1$ let $R = \text{diag}(-1,1,\dots ,1) \in \OC{n}$ and set $g = RG \in \SOC{n}$. Since $R$ is unitary we have that
	$$ (RGu)^*(RGv) = (Gu)^*R^*R(Gv) =(Gu)^*Gv=0.$$

\end{proof}

\subsection{Critical points in the fibre over $0$}
In this section we show that the differential of a $\mu$-eigenfunction must have rank $0$ or $2$ on the fibre over $0$. This is a known result, but is presented here for the convenience of future readers.

\begin{proposition}
	Let $\phi: M \to \cn$ be an eigenfunction of the conformality operator $\kappa$ such that $0$ is in the image of $\phi$. Then $0$ is a regular value of $\phi$ so long as the (complexified) gradient $\nabla \phi$ does not vanish on $\phi^{-1}(\{0\})$.
\end{proposition}
\begin{proof}
	First notice that the rank of $d\phi$ is equal to the dimension of the span of $\nabla u,\nabla v$ where $\phi = u + iv$.
	
	It remains to show that $d\phi$ cannot be rank-one at any point of $\phi^{-1}(\{0\})$. Suppose, towards a contradiction, that $d\phi_p$ has rank-one and thus $\nabla u_p$ and $\nabla v_p$ are proportional. If $\nabla u_p \neq 0$ we can compute
	\begin{eqnarray*}
		\kappa(\phi,\phi)(p) &=& g(\nabla u + i \cdot  \nabla v, \nabla u + i \cdot \nabla v)_p \\
		&=& g(\nabla u + i \cdot c \cdot  \nabla u, \nabla u + i \cdot c \cdot \nabla u)_p \\
		&=& (1-c^2)|\nabla u|_p^2 + 2i\cdot c\cdot  |\nabla u|_p^2 \neq 0.
	\end{eqnarray*}
	If $\nabla u_p = 0$, then we have that 
	\begin{eqnarray*}
		\kappa(\phi,\phi)(p) &=& g(i \cdot  \nabla v, i \cdot \nabla v)_p \\
		&=& - |\nabla v|_p^2 \neq 0.
	\end{eqnarray*}
	However, since $\phi$ is eigen,
	$$ \kappa(\phi,\phi)(p) = \mu \cdot (\phi(p))^2 = 0$$
	so we arrive at a contradiction. 
\end{proof}

The following lemma describes the points $z$ at which the gradient of an eigenfunction vanishes. In particular, it describes all the possible critical points in the fibre $\phi^{-1}(\{0\})$.
\begin{lemma}\label{lemma-critical-point}
	Let $\psi: H^\cn/H \to \cn$ be an eigenfunction given by
	$$ \psi(z) = \trace (A z z^*).$$
	Then  $\nabla\psi (z) = 0$ if and only if 
	$$ \langle z^* A z, Z \rangle =0$$ 
	for all $ Z \in \h^\cn$.
\end{lemma}
\begin{proof}
	First recall that for matrix representations we have that $Z(z) = z Z$ and $Z(z^*) = (z Z)^*$.
	Since $Z \in B_\k \cup i \cdot B_\m$ in $\h \oplus i \cdot \h$ we can compute
	\begin{eqnarray*}
		Z( \trace (A z z^*)) &=&  \trace (A Z(z) z^*) +  \trace (A z Z(z^*)) \\
		&=&  \trace (A z Z z^*) +  \trace (A z Z^* z^*) \\
		&=&  \trace (A z (Z+Z^*) z^*) \\
		&=&  \trace ( z^* A z (Z+Z^*)).
	\end{eqnarray*}
	
	Since $\h$ is one of $\so{n}, \su{n}$ or $\sp{n}$ we have that for any $X \in B_\k$, $X^* = -X$ and so the above equation vanishes.
	Similarly for $i \cdot Y \in i \cdot B_\m$ we have $(iY)^* = -iY^* = iY$ so we get that
	
	\begin{eqnarray*}
	(iY)( \trace (A z z^*)) &=& 2i \cdot \trace (z^* A z Y) \\
	&=& 2i (\Re \trace (z^* A z Y) + i \Im \trace (z^* A z Y)) \\
	&=& -2i (\Re \trace (z^* A z Y^*) + i \Im \trace (z^* A z Y^*)) \\
	&=& -2i (\Re \trace (z^* A z Y^*) + i \Re \trace (z^* A z (iY)^*)) \\
	&=& -2i ( \langle z^* A z,Y \rangle + i \langle z^* A z, iY \rangle)
	\end{eqnarray*}
which vanishes if and only if
$$ \langle z^* A z,Y \rangle = \langle z^* A z, iY \rangle = 0, $$
i.e.,
$$  z^* A z \perp d\rho(\h^\cn).$$
\end{proof}

\section{The Riemannian Lie Group $\GLC n$}

The main purpose of this short section is to introduce some useful notation.  The complex general linear group is given by 
$$\GLC{n}=\{z\in\cn^{n\times n}\ |\ \det z \neq 0\}.$$
Its Lie algebra $\glc n$ of left-invariant vector fields can be identified with the tangent space at the identity element $e\in\GLC n$, namely the $n\times n$ complex matrices in $\cn^{n\times n}$.  We equip $\GLC n$ with its standard Riemannian metric induced by the Euclidean scalar product
on the Lie algebra $\glc n$ given by
$$g(Z,W)=\Re\trace (ZW^*).$$  
For $1\le i,j\le n$ we denote by $E_{ij}$ the element of $\glr n$ satisfying
$$(E_{ij})_{kl}=\delta_{ik}\delta_{jl}$$ 
and by $D_t$ the diagonal matrix $D_t=E_{tt}$. For $1\le r<s\le n$ let $X_{rs}$, $Y_{rs}$ and $D_{rs}$ be the matrices satisfying
$$X_{rs}=\frac 1{\sqrt 2}(E_{rs}+E_{sr}),\ \ Y_{rs}=\frac
1{\sqrt 2}(E_{rs}-E_{sr}), \ \ D_{rs} = \frac{1}{\sqrt 2} (E_{rr}-E_{ss}).$$

\section{The Real Cases $\SO n$ and $\SOC n/\SO n$}
\label{section-real-cases}

The Lie algebra $\so n$ of $\SO n$ is the set of real  skew-symmetric matrices
$$\so n=\{X\in\glr n\ |\ X+X^*=0\}$$ 
with canonical orthonormal basis
$$\B_{\so n}=\{Y_{rs}\ |\ 1\le r<s\le n\}.$$
Its complexified Lie algebra $\soc{n}$ consists of all complex skew-symmetric matrices, so its orthogonal complement in $\glc{n}$ is the space of all complex symmetric matrices.
The gradient $\nabla\phi$ of a complex-valued function $\phi:\SO n\to\cn$, is an element of the complexified tangent bundle $T^\cn\SO n$.  This satisfies 
$$\nabla\phi=\sum_{Y\in\B_{\so n}}Y(\phi)\cdot Y.$$

The simplest examples of eigenfamilies on $\SO{n}$ come from the standard representation.
\begin{theorem}[\cite{Gud-Sak-1}]
	Let $V$ be a maximal isotropic subspace of $\cn^n$ and $p \in \cn^n$ be a non-zero element. Then the set 
	$$ \E_V(p) = \{ \phi_a : \SO{n} \to \cn \mid \phi_a (x)= \trace (p a^* x), \ a\in V \}$$
	of complex-valued functions is a $(-\tfrac{(n-1)}{2},-\tfrac{1}{2})$-eigenfamily on $\SO{n}$.
	\end{theorem}
By duality, we have that the set
$$ \F_V(p) = \{ \psi_a : \SOC{n}/\SO{n} \to \cn \mid \psi_a(z) = \trace (p a^* {z} z^{*}), \ a\in V \}$$
 is a  $(\tfrac{(n-1)}{2},\tfrac{1}{2})$-eigenfamily on $\SOC{n}/\SO{n}$.
 
\begin{remark}
	Before proceeding with the proof of Theorem \ref{thm-so-minimal}, we first consider the degenerate case $n=2$. With respect to the standard basis of $\cn^2$ the two isotropic directions are spanned by the vectors
$$ e_+ = \begin{pmatrix}
	1 \\ i
\end{pmatrix} \text{ and } e_- = \begin{pmatrix}
	1 \\ -i
\end{pmatrix}.$$
Notice that $e_+$ and $e_-$ are Hermitian-orthogonal and $e^t_+ e_- = 2$.
After rescaling, we can assume that $a = e_+$ and  $p = \alpha \cdot e_+ + \beta \cdot e_-$ where 
$\alpha,\beta \in \cn \setminus \{0\}$. Since the action of $\SOC{2}$ preserves isotropic lines we have that for any $z \in \SOC{2}$
$$ z e_+ = \gamma \cdot e_+, \ \ z (\alpha \cdot e_+ + \beta \cdot e_-) = \gamma \alpha \cdot e_+ + \gamma^{-1} \beta \cdot e_-$$
for some non-zero $\gamma$, and so
$$ \psi_a(z) = (z^* a)^* (z^* p) = 2 \alpha |\gamma|^2 \neq 0.$$
Note that if we allow $p$ to be isotropic, then $\psi_a$ is nowhere-zero, or the function is identically zero when $p \in \cn e_+$ or $p\in \cn e_-$ respectively, so $\psi^{-1}(\{0\})$ will either be the empty set or all of $\SOC{2}/\SO{2}$.
\end{remark}

\begin{proof}[Proof. (Theorem \ref{thm-so-minimal})]
Since $a$ is isotropic and $p$ is not, they are linearly independent and so Lemma \ref{lemma-zeros-sonC} ensures that there exists $z \in \SOC{n}$ such that $\psi_a(z) = 0$. 

By Lemma \ref{lemma-critical-point} we have that $z \in \SOC{n}$ is a critical point of $\psi_a = \trace (p a^* {z} z^{*})$ if
$$ \langle z^* p a^* z, Z \rangle = 0$$
for all $Z \in \soc{n}$.
In particular, this only occurs when $z^* p a^* z \in \soc{n}^\perp$, i.e. whenever $z^* p a^* z$ is symmetric. Since $z$ and $z^*$ are invertible, they have full rank, and $z^* p a^* z$ has rank-one at each point of $\SOC{n}$, so there exists $u,v \in \cn^n$ such that
\begin{eqnarray*}
	z^* p a^* z &=& u v^t \\
	&=& (u v^t)^t \\
	&=& v u^t
\end{eqnarray*}
so $u$ and $v$ are linearly dependent, $v = c \cdot u$.
If $\psi_a(z) = 0$, then we additionally require that $z^* p a^* z = c\cdot u u^t$ be trace-free, so $u$ must be an isotropic vector. In particular, since whenever $ a b^t = c d^t$ we have that the pairs of vectors $(a,c)$ and $(b,d)$ must be proportional to each other, and since the orbit of an isotropic vector under $\SOC{n}$ consists only of isotropic vectors, we must have that both $p$ and $a$ are isotropic, which contradicts the assumption that $p$ is non-isotropic.

\end{proof}

\section{The Complex Cases $\SU n$ and $\SLC n/\SU n$}
\label{section-complex-cases}

The Lie algebra $\su n$ of $\SU n$ is the set of the trace-free complex skew-Hermitian matrices
$$\su n=\{Z\in\glc n |\ Z+ Z^*=0, \ \trace Z = 0\}$$ 
 which is spanned by the set
$$ \B_{\su{n}} = \{ Y_{rs}, i X_{rs}  \mid 1 \leq r  < s \leq n\} \cup \{ iD_{1t} \mid 1 < t \leq n\}.$$ 

Its complexified Lie algebra $\slc{n}$ consists of all $n \times n$ trace-free complex matrices, which has orthogonal complement spanned by the identity matrix $I_n$ in $\glc{n}$.

The simplest examples of eigenfamilies on $\SU{n}$ come from the standard representation.
\begin{theorem}[\cite{Gud-Sak-1}]
Let $p \in \cn^n$ be a non-zero element. Then the set 
	$$ \E(p) = \{ \phi_a : \SU{n} \to \cn \mid \phi_a(z)= \trace (p a^* z), \ a\in \cn^n \}$$
	of complex-valued functions is a $(-\tfrac{(n^2-1)}{n},-\tfrac{n-1}{n})$-eigenfamily on $\SU{n}$.
	\end{theorem}
	
Thus we obtain an eigenfamily
$$ \F(p) =\{\psi_a : \SLC{n}/\SU{n} \to \cn \mid \psi_a = \trace (p a^* {z} z^*), \ a \in \cn^n\}.$$

\begin{proof}[Proof. (Theorem \ref{thm-slc-minimal})]
	First we show that $0$ is in the image of $\psi_a$ if and only if $p$ and $a$ are linearly independent. The computation at the start of subsection \ref{subsection-zeros} shows the case when the vectors are dependent. Now suppose that $a$ and $p$ are linearly independent and extend them to a basis of $\cn^n$
	$$ \{ a,p,v_3, \dots v_n \}.$$
	Let $$D=det(a,p,v_3,\dots,v_n) \neq 0.$$ 
	Then define $g$ to be the unique linear map that sends this basis to 
	$$ \{ D\cdot e_1, e_2, \dots e_n\},$$
	where $e_i$ form an orthonormal basis for the standard Hermitian inner product of $\cn^n$.
	Then we see that 
	$$ \det g = \frac{\det (D\cdot e_1, \dots e_{n-1},e_n)}{\det(a,p,v_3,\dots,v_n)} = \frac{D}{D} = 1$$
	so $g \in \SLC{n}$. Then if we let $ z = g^*$ we have that
	$$\psi_a(z) = a^* z z^* p = (ga)^* (gp) = \bar D \cdot e_1^* e_2 = 0.$$
	
By Lemma \ref{lemma-critical-point} we have that the gradient of $\psi_a = \trace (p a^* {z} z^{*})$ vanishes at $z \in \SLC{n}$   if
$$ \langle z^* p a^* z, Z \rangle = 0$$
for all $Z \in \slc{n}$. In particular this only occurs when $z^* p a^* z \in \slc{n}^\perp = \{ c \cdot I_n \}$. Since $z$ and $z^*$ are invertible, they have full rank and therefore $z^* p a^* z$ has rank-one at each point of $\SLC{n}$, while $I_n$ is full rank. So $\nabla \psi_a(z)$ never vanishes, and in particular $\psi_a$ has no critical points in the fibre over $0$.
\end{proof}

\begin{example}
Suppose that $p = e_i, a = e_j$ then we have that
$$ \psi_{ij} = \trace (e_i e_j^t  {z} z^*) = (z z^*)_{ji} = \langle z_j, z_i \rangle_\cn $$
where $z_j,z_i$ denote the $j$-th and $i$-th rows of $z$ respectively. Notice that when $i = j$ we get $|z_i|^2 \neq 0$ for all $z \in \SLC{n}$ and whenever $i \neq j$ we have that $\psi_{ij}(I) = 0$.

\end{example}

\section{The Quaternionic Cases $\Sp n$ and $\SpC {n}/\Sp n$}
\label{section-quaternionic-cases}
The Lie algebra $\sp n$ of $\Sp n$ satisfies
$$\sp{n}=\{\begin{pmatrix} Z & W
\\ -\bar W & \bar Z\end{pmatrix}\in\cn^{2n\times 2n}
\ |\ Z^*+Z=0,\ W^t-W=0\}.$$ 

For the indices $1\le r<s\le n$ and $1\le t\le n$ we now  introduce the following notation for the elements of the orthonormal basis $\B_{\sp n}$ of the Lie algebra $\sp n$ of the quaternionic unitary group $\Sp n$:
$$Y^a_{rs}=\frac 1{\sqrt 2}
\begin{pmatrix}
Y_{rs} & 0 \\
     0 & Y_{rs}
\end{pmatrix},\  
X^a_{rs}=\frac 1{\sqrt 2}
\begin{pmatrix}
iX_{rs} & 0 \\
      0 & -iX_{rs}
\end{pmatrix},$$
$$ X^b_{rs}=\frac 1{\sqrt 2}
\begin{pmatrix}
      0 & iX_{rs} \\
iX_{rs} & 0\end{pmatrix},\ 
X^c_{rs}=\frac 1{\sqrt 2}
\begin{pmatrix}
      0 & X_{rs} \\
-X_{rs} & 0
\end{pmatrix},$$
$$D^a_{t}=\frac 1{\sqrt 2}
\begin{pmatrix}
iD_{t} & 0 \\
     0 & -iD_{t}
\end{pmatrix},
D^b_{t}=\frac 1{\sqrt 2}
\begin{pmatrix}
     0 & iD_{t}  \\
iD_{t} & 0
\end{pmatrix},
D^c_{t}=\frac 1{\sqrt 2}
\begin{pmatrix}
     0 & D_{t}  \\
-D_{t} & 0
\end{pmatrix}.$$
\smallskip

The simplest examples of eigenfamilies on $\Sp{n}$ come from the standard representation.
\begin{theorem}[\cite{Gud-Sak-1}]
Let $p \in \cn^n$ be a non-zero element. Then the set 
	$$ \E(p) = \{ \phi_{ab} : \Sp{n} \to \cn \mid \phi_{ab}(g)= \trace (p a^* z^t+pb^*w^t), \ a,b\in \cn^n \}$$
	of complex-valued functions is a $(-\tfrac{(2n+1)}{2},-\tfrac{1}{2})$-eigenfamily on $\Sp{n}$.
\end{theorem}
Notice that this is an eigenfunction of the form 
\begin{eqnarray*}
\trace (p a^* z^t+pb^*w^t) &=& \trace (p a^* z^t) + \trace (pb^*w^t) \\
&=&  \trace (z (p a^*)^t) + \trace (w (pb^*)^t) \\
&=& \trace ( \bar ap^t z) + \trace (\bar bp^t w) \\
&=& \trace \left( \begin{pmatrix}
	\bar a \\ \bar b
\end{pmatrix} \begin{pmatrix}
p^t & 0
\end{pmatrix}  \begin{pmatrix}
z & w\\ -\bar w & \bar z
\end{pmatrix}\right)
 \\
&=& \trace (\tilde A q)
\end{eqnarray*}
where $$\tilde A = \begin{pmatrix}
	\bar a \\ \bar b
\end{pmatrix} \begin{pmatrix}
p^t & 0
\end{pmatrix} .$$
To simplify notation we relabel each of $a,b$ and $p$ by its complex conjugate.

Recall that the standard representation of $\SpC{n}$ is given by matrices of the form
$$ q=\begin{pmatrix}
	x & y \\
	z & w
\end{pmatrix}$$
where
\begin{equation}\label{eqn-Sp2nC}
	x^tz=z^tx,w^ty=y^tw \text{ and } x^tw-z^ty = I_n.
\end{equation}

Thus we have that the set \small
$$  \F(p) = \left\{ \psi_{ab}:\SpC{n}/\Sp{n} \to \cn \mid \psi_{ab} = \trace\left( \begin{pmatrix}
	a \\ b
\end{pmatrix} \begin{pmatrix}
p^* & 0
\end{pmatrix}  q q^*\right), \ a,b \in \cn^n
\right\} $$ \normalsize
is a $(\tfrac{2n+1}{2},\tfrac{1}{2})$-eigenfamily on $\SpC{n}/\Sp{n}$.

\begin{lemma}\label{lemma-SP2nC-Submersion}
Let $A$ be a rank-one matrix. Any eigenfunction of the form 
$$ \psi=\trace (A qq^*) :\SpC{n}/\Sp{n} \to \cn$$
has a nowhere-vanishing gradient.
\end{lemma}
\begin{proof}
Once again we will apply Lemma \ref{lemma-critical-point}. Recalling that
$$ \mathfrak{sp}_{n}(\cn) = \{ X \in \glc{2n} \mid X^tJ+JX = 0 \}, $$	
where $$ J = \begin{pmatrix}
	0 & I \\
	-I & 0
\end{pmatrix} $$
we can compute $\mathfrak{sp}_{n}(\cn)^\perp \subset \glc{2n}$ by considering the involution
$$ \sigma(X) = JX^tJ$$
whose $\pm 1$ eigenspaces are the spaces where $JX$ is symmetric and skew-symmetric respectively.
By the skew-symmetry of $J$ and the fact that $J^2 =-I_{2n}$ we have that:
\begin{eqnarray*}
	\langle \sigma(X),\sigma(Y) \rangle &=& \Re \trace ( \sigma(X) \sigma(Y)^*) \\
	&=& \Re \trace ( JX^t J J^* {Y^*}^t J^*) \\
	&=& \Re \trace ( X {Y^*})\\
\end{eqnarray*}
Thus $\sigma$ is an isometry, so the eigenspaces are orthogonal and therefore 
$$ \mathfrak{sp}_{n}(\cn)^\perp = \{ Y \in \glc{2n} \mid Y^tJ-JY = 0 \}.$$	
By Lemma \ref{lemma-critical-point}, $q$ is a critical point of $\psi$ if and only if 
$$ q^* A q \in \mathfrak{sp}_{n}(\cn)^\perp,$$
however $q^* A q$ is always rank-one, while every $Y \in \mathfrak{sp}_{n}(\cn)^\perp$ has the same rank as $JY$ which must be even, since $JY$ is skew-symmetric.
\end{proof}
In order to produce minimal submanifolds from these eigenfunctions we need to have $0$ in the image.

Now we can prove the main result of this section.

\begin{proof}[Proof. (Theorem \ref{thm-Sp-minimal})]
	We first suppose that $p$ and $a$ are linearly independent. Choose a Hermitian orthonormal basis $\{e_1,\dots, e_n\}$ for $\cn^n$ such that $p,a,b$ are contained in the spans of $\{e_1\},\{e_1,e_2\},\{e_1,e_2,e_3\}$ respectively. We first consider the case that $p$ and $a$ are linearly independent.
	Let $\alpha \in \cn$ and consider the matrix
	$$ q =\begin{pmatrix}
		I_n & \begin{pmatrix}
		0  & 1 &0 \\ 1  & \alpha & 0 \\ 0 & 0 & 0_{n-2}
		\end{pmatrix} \\ 0 & I_{n}
	\end{pmatrix} $$
	which clearly satisfies equations (\ref{eqn-Sp2nC}) and so is a point of $\Sp{n,\cn}$.
	Now we compute
	\begin{eqnarray*}
		&&\psi_{ab}(q) \\&=& \trace\left( \begin{pmatrix}
	a \\ b
\end{pmatrix} \begin{pmatrix}
p^* & 0
\end{pmatrix}  q q^*\right) \\
&=& \trace\left( \begin{pmatrix}
	a \\ b
\end{pmatrix} \begin{pmatrix}
p^* & 0
\end{pmatrix}  \begin{pmatrix}
		I_n & \begin{pmatrix}
		0  & 1 &0 \\ 1  & \alpha & 0 \\ 0 & 0 & 0_{n-2}
		\end{pmatrix} \\ 0 & I_{n}
	\end{pmatrix} 
	\begin{pmatrix}
		I_n & 0 \\ \begin{pmatrix}
		0  & 1 &0 \\ 1  & \overline{\alpha} & 0 \\ 0 & 0 & 0_{n-2}
		\end{pmatrix} & I_{n}
	\end{pmatrix} \right) \\
	&=& \begin{pmatrix}
p^* & 0
\end{pmatrix} 
\begin{pmatrix}
		I_n+ \begin{pmatrix}
		1 & \overline{\alpha} &0 \\ \alpha  & 1+\alpha \overline{\alpha} & 0 \\ 0 & 0 & 0_{n-2}
		\end{pmatrix}  & \begin{pmatrix}
		0 & 1 &0 \\ 1  & \alpha & 0 \\ 0 & 0 & 0_{n-2}
		\end{pmatrix} \\ \begin{pmatrix}
		0  & 1 &0 \\ 1  & \overline{\alpha} & 0 \\ 0 & 0 & 0_{n-2}
		\end{pmatrix} & I_{n} 
	\end{pmatrix} 
\begin{pmatrix}
	a \\ b
\end{pmatrix} \\
&=& \bar p_1(2a_1+\overline{\alpha}a_2+b_2)
	\end{eqnarray*} 
	which vanishes if and only if 
	$$2a_1+\overline{\alpha}a_2+b_2 = 0.$$
	Since $a$ is not a scalar multiple of $p$, $a_2 \neq 0$ so we choose $\alpha = \overline{ \frac{-2a_1-b_2}{a_2}.}$

Now suppose that $p$ and $a$ are linearly dependent, i.e $ a = \gamma \cdot p$. If $b = 0$ the calculation in subsection \ref{subsection-zeros} shows that $\psi_{ab}$ is never zero, so it remains to consider the case $b \neq 0$. Choose a basis $\{e_1,\dots, e_n\}$ for $\cn^n$ such that $p$ and $a$ are proportional to $e_1$ and $b$ is contained in the span of $\{e_1,e_2\}$. In this basis the matrix $A$ has representation

$$ \begin{pmatrix}
\gamma & 0 & \cdots & 0 \\
0 & 0 & \cdots & 0 \\
\vdots & & \vdots &\\
0 & 0 & \cdots & 0 \\
\mu & 0 & \cdots & 0 \\
\nu & 0 & \cdots & 0 \\
0 & 0 & \cdots & 0 \\
\vdots &  & \vdots & \\
0 & 0 & \cdots & 0 
\end{pmatrix}  $$
for some $\gamma,\mu,\nu \in \cn$ with not both $\mu$ and $\nu$ being $0$. 
Thus we have that
$$\psi_{ab}(q) = \trace(A qq^*) = \gamma \cdot (q_1,q_1)+\mu\cdot(q_1,q_{n+1})+\nu \cdot (q_1,q_{n+2})$$
where $(\cdot,\cdot)$ denotes the Hermitian inner product on $\cn^{2n}$ and $q_i$ is the $i$-th row of $q$.
Then notice that for $q_0 \in \SpC{n}$ of the form
$$ q_0=\begin{pmatrix}
	I & 0 \\
	y & I
\end{pmatrix}$$
where we let $$y= -\frac{\bar \gamma}{\bar \mu}E_{11} \ \  \text{or} \ \  y= - \frac{\bar \gamma}{\bar \nu} (E_{12}+E_{21})$$
we have that $\psi_{ab}(q_0) = 0$. Then by Lemma \ref{lemma-SP2nC-Submersion} and Theorem \ref{Gud-Mun-1-Thm} we have that the fibre over $0$ is a minimal submanifold of codimension two.
\end{proof}


                 

\end{document}